\documentclass[12pt]{amsart}

\usepackage[utf8]{inputenc}

\usepackage[utf8]{inputenc}
\usepackage{amssymb}
\usepackage{graphicx}
\usepackage{graphicx,color}
\usepackage{latexsym}
\usepackage[all]{xy}
\usepackage{mathrsfs}
\usepackage{enumerate}
\usepackage{multicol}
\usepackage{lipsum}
\usepackage{amssymb}
\usepackage{tikz}
\usepackage{float}
\usepackage{bm}
\usepackage{mathptmx}
\usetikzlibrary{shapes.geometric}

\usetikzlibrary{arrows}
\usepackage{amsmath}
\usepackage{amsfonts}
\usepackage{amssymb,enumerate}
\usepackage{amsthm}
\usepackage{fancyhdr}
\usepackage{hyperref}
\usepackage{cleveref}
\usepackage{xcolor}
\usepackage{enumitem}
\usepackage{float}
\usepackage{verbatim}
\usepackage{tikz}
\usetikzlibrary[patterns]
\usetikzlibrary{arrows}

\usepackage{mathrsfs}
\usetikzlibrary{arrows}
\usepackage{listings}
\usepackage{pgfplots}
\pgfplotsset{compat=1.15}

\usepackage{siunitx}

\usepackage{mathrsfs}
\definecolor{DimGray}{rgb}{0.41, 0.41, 0.41}
\definecolor{zzttqq}{rgb}{0.6,0.2,0.}
\definecolor{uuuuuu}{rgb}{0.26666666666666666,0.26666666666666666,0.26666666666666666}
\definecolor{ududff}{rgb}{0.30196078431372547,0.30196078431372547,1.}
\definecolor{xdxdff}{rgb}{0.49019607843137253,0.49019607843137253,1.}
\newtheorem{theorem}{Theorem}[section]
\theoremstyle{definition}
\newtheorem{defin}[theorem]{Definition}
\newtheorem{question}[theorem]{Question}

\newtheorem{ex}[theorem]{Example}
\theoremstyle{plain}
\newtheorem{proposition}[theorem]{Proposition}
\newtheorem*{theorem*}{Theorem}
\newtheorem{lemma}[theorem]{Lemma}
\newtheorem{corollary}[theorem]{Corollary}
\newtheorem{conj}[theorem]{Conjecture}

\theoremstyle{remark}
\newtheorem{rem}[theorem]{Remark}

\numberwithin{equation}{section}
\numberwithin{figure}{section}

\title[Wilf's conjecture for semimodules]{An extension of the Wilf conjecture to semimodules over a numerical semigroup}
\author{Patricio Almir\'on}
\author{Julio-Jos\'e Moyano-Fern\'andez}



\thanks{The first author was partially supported by Spanish Goverment, Ministerios de Ciencia e Innovaci\'on y de Universidades MTM2016-76868-C2-1-P. The second author was partially supported by the Spanish Government, Ministerios de Ciencia e Innovaci\'on y de Universidades, grant PGC2018-096446-B-C22, as well as by Universitat Jaume I, grant UJI-B2018-10. Both authors contributed equally to this work.}

\address{Instituto de Matemática interdisciplinar (IMI) y departamento de \'{A}lgebra, Geometr\'{i}a y Topolog\'{i}a\\
Facultad de Ciencias Matem\'{a}ticas\\
Universidad Complutense de Madrid\\
28040, Madrid, Spain.}

\email{palmiron@ucm.es}

\address{Universitat Jaume I, Campus de Riu Sec, Departamento de Matem\'aticas \& Institut Universitari de Matem\`atiques i Aplicacions de Castell\'o, 12071
Caste\-ll\'on de la Plana, Spain}

\email{moyano@uji.es}

\begin{document}

\begin{abstract}
The aim of this paper is to propose an extension of the Wilf conjecture to semimodules over a numerical semigroup through a new approach toward the solution of the Wilf conjecture on numerical semigroups. 
The key point is the introduction of a new invariant that we call Wilf function of a semigroup that can be also defined for a semimodule over the semigroup. We also study the relation of the Wilf function of a semigroup and the Wilf function of certain semimodules associated to a gap of the semigroup. As a consequence, we obtain an alternative conjecture which also implies the Wilf conjecture. Finally, we show that the natural generalization of Wilf conjecture to semimodules over the semigroup does not work unless we used the Wilf function of the semimodule. In this direction, we propose several questions regarding the possible extensions of Wilf-type inequalities for semimodules.
\end{abstract}

\maketitle

\section{Introduction}

The ``money-changing problem'' asks for those sums of money we can change if we have coins of \(e\) different values, say \(a_1,\dots, a_e\). In 1978, H. Wilf \cite{wilf} presented an algorithm to solve this problem; he named it ``circle-of-lights algorithm''. The idea is quite basic: we consider a circle of \(a_e\) lights labeled as \(0,1,\dots,a_e\), where all of them are ``off'' except for the one representing \(0\). We start turning on lights clockwise. During the first round we turn on the lights corresponding to the values \(\{a_1,a_2,\dots,a_e\}\). After that, the algorithm runs turning on bulbs representing the possible combinations of numbers \(x=x_0a_0+\cdots+ x_ea_e\). Observe that the representation of \(x\) is given modulo \(a_e\). If the possible values of coins \(\{a_1,a_2,\dots,a_e\}\) are coprime, then the algorithm terminates after a finite number of steps. In this case the ``money-changing problem'' can be easily translated in terms of numerical semigroups.
\medskip

A numerical semigroup $\Gamma$ is an additive submonoid of the natural numbers of finite complement. The elements in this complement are said to be the gaps of the semigroup, and the number $g(\Gamma)$ of all of them is called the genus of $\Gamma$. The fact that $g(\Gamma)<\infty$ implies that $\Gamma$ is finitely generated, and it is not difficult to see that every numerical semigroup has a unique system of minimal generators; its cardinality is called the embedding dimension of $\Gamma$, written $e(\Gamma)$. Observe that in the ``circle-of-lights algorithm'' the minimal generators of the semigroup correspond to coin values \(\{a_1,\dots,a_e\}\) and the bulbs that are ``off'' correspond to the gaps of the semigroup generated by the possible values of coins (see \cite[II, p. 563]{wilf}). The biggest gap with respect to the usual ordering is called the Frobenius number of $\Gamma$, say $F(\Gamma)$, and $c(\Gamma):=F(\Gamma)+1$ is called the conductor of $\Gamma$. Finally, it is also customary to consider the delta-invariant of $\Gamma$, which is defined to be the cardinality $\delta(\Gamma)$ of the set $\{x\in \Gamma : x<F(\Gamma)\}$. 
\medskip

Numerical semigroups became notorious in Mathematics due to the Frobenius problem, which asks for a closed formula for the Frobenius number or, equivalently, for the conductor of a numerical semigroup. In particular, the \textquotedblleft circle-of-lights algorithm\textquotedblright ~gives a partial answer to Frobenius problem, i.e. given a set of generators of the semigroup this is an algorithm which computes the conductor of the semigroup (see \cite[p. 564]{wilf}). Moreover, the \textquotedblleft circle-of-lights algorithm\textquotedblright ~ computes \(c(\Gamma)\) with time complexity \(O(e(\Gamma)c(\Gamma))\) (see \cite[p. 833]{ninjen}). Based on this complexity H. Wilf \cite{wilf} posed the following question: is there a general upper bound for the density of the gaps of a numerical semigroup in the integer interval bounded by $0$ and the Frobenius number? This can be formulated as follows \cite{fgh}:

\begin{conj}[Wilf conjecture]\label{Wilfconjecture}
For any numerical semigroup $\Gamma$, its conductor, embedding dimension, number of gaps and delta-invariant are related by means of the inequalities
\begin{equation}\label{ineqn:wilf}
\frac{|\mathbb{N}\setminus\Gamma|}{c(\Gamma)}\leq 1-\frac{1}{e(\Gamma)},\quad\text{or equivalently}\quad	c(\Gamma)\leq e(\Gamma)\cdot \delta(\Gamma).
\end{equation}
\end{conj}

Many particular cases of this conjecture are known, see e.g. Delgado, \cite{delgado1}, Dobbs and Matthews \cite{dm}, Eliahou \cite{eliahou}, Kaplan \cite{kaplan}, Sammartano \cite{samartano}, Bruns, Garc\'ia-S\'anchez, O'Neil and Wilburne \cite{bruns}, among others. For a recent account of the conjecture we refer the reader to the survey \cite{mdelgado}.
\medskip

A. Nijenhuis \cite{ninjen} proposed an alternative method that improves the time complexity of Wilf's algorithm. More recently, S. B\"{o}cker and Z. Lipt\'{a}k \cite{proteinas} improved the time complexity of Nijenhuis' algorithm to \(O(e(\Gamma)a_1)\) with interesting applications to interpreting mass spectrometry peaks. Based on those improvements as well as on the huge number of examples provided by many of the above mentioned references \cite{dgr,delgado1,eliahou,kaplan,dm}, one realizes that the inequality in the Wilf conjecture \ref{Wilfconjecture} seems to be far from being tight. Having this in mind, the underlying meaning of the Wilf conjecture reveals to be the finding of an upper bound for the quotient \(c(\Gamma)/\delta(\Gamma)\) and, after that, the comparison of this bound with the embedding dimension of the semigroup.
 \medskip
 
This idea leads us to not directly consider the Wilf conjecture but a general inequality of the type \(k\delta(\Gamma)>c(\Gamma)\) with \(k\in\mathbb{Z}\); we call this inequality a \emph{Wilf-type inequality}. In this general setting, it makes sense to define the \emph{Wilf function} associated to a numerical semigroup as the map $W_{\Gamma}:\mathbb{N}\to \mathbb{Z}$ given by
$$
k\mapsto k\delta(\Gamma)-c(\Gamma).
$$

For $k=e(\Gamma)$, the value $W_{\Gamma}(e(\Gamma))$ is called the \emph{Wilf number} of $\Gamma$, see \cite{eliahou}, \cite{almiyano-tetris}; also \cite[p.~45]{mdelgado}. Of course, Conjecture \ref{Wilfconjecture} may be trivially rewritten in terms of the Wilf number as $W_{\Gamma}(e(\Gamma))\geq 0$. Moreover, by the previous discussion, it becomes natural to try to find the best upper bound, say \(\mu_{\Gamma}\) for the quotient \(c(\Gamma)/\delta(\Gamma)\) (see end of Section \ref{sec:wilffunction}).
\medskip

On the other hand, inspired by Fr\"oberg, Gottlieb and H\"aggkvist \cite{fgh}, Moscariello and Sammartano \cite{mossam} asked for the equality in (\ref{ineqn:wilf}), and they proposed the following conjecture:

\begin{conj}[Fr\"ogoh\"amosa-conjecture]\label{alesios}
Let $\Gamma\neq \mathbb{N}$ be a numerical semigroup.
The equality $c(\Gamma)= e(\Gamma)\cdot \delta(\Gamma)$ holds if and only if $\Gamma$ has embedding dimension $2$ or there exist $m,q\in \mathbb{N}\setminus \{0\}$ with $m>1$ such that
$$
\Gamma=W_{m,q}:=\{0, m, 2m, 3m, \ldots, (q-1)m, qm, qm+1, qm+2, \ldots\}.
$$
\end{conj}
Observe that in the case \(\Gamma=\mathbb{N}\) the equality holds trivially. Also, for \(\Gamma\neq \mathbb{N}\), it is easy to check that the semigroups of the form \(\Gamma=\langle a,b \rangle\) and \(\Gamma=W_{m,q}\) one can easily check that they satisfy the equality  $c(\Gamma)= e(\Gamma)\cdot \delta(\Gamma)$; therefore the difficult part is the converse statement. Although many numerical experiments have been done, there is no hint to prove Conjecture \ref{alesios} in its whole generality. Moreover, there is no philosophical reason explaining why the semigroups occurring in Conjecture \ref{alesios} are exactly those. One of the contributions in this paper is then to propose an explanation for the reason why the semigroups involved in the Fr\"ogoh\"amosa-conjecture are precisely the ones appearing and no others.
\medskip

Thanks to the introduction of the Wilf function we are able to prove the following:
\begin{theorem*}[Theorem \ref{thm:extreme}]
	Let \(\Gamma\) be a numerical semigroup. Then,
	\begin{enumerate}
		\item \(\Gamma=\mathbb{N}\) if and only if \(W_\Gamma(k)\geq 0\) for \(1\leq k\leq m\).
		\item \(\Gamma=W_{m,q}\) for \(q\geq 1\) if and only if \(W_{\Gamma}(k)\leq 0\) for all \(1\leq k\leq m\).
		\item \(\Gamma=\langle a,b \rangle\) with \(\gcd(a,b)=1\) is the numerical semigroup with minimal embedding dimension between those satisfying \(W_{\Gamma}(k)\geq 0\) for all \(2\leq k\leq m\) and \(W_{\Gamma}(2)=0\).
	\end{enumerate}
\end{theorem*}
In fact, it was known that, under the assumption of maximal embedding dimension for a semigroup, then \((2)\) is true. This is not directly related to the statement \((2)\) of Theorem \ref{thm:extreme}, since we make no assumptions about the embedding dimension of the semigroup; this means that the condition \(W_{\Gamma}(k)\leq 0\) for all \(2\leq k\leq m\) implies maximal embedding dimension. In this way, Theorem \ref{thm:extreme} allows us to say that the semigroups appearing in the Fr\"ogoh\"amosa-conjecture are exactly those for which the Wilf function \(W_\Gamma(k)\) has an extreme behaviour, where \emph{extreme} is precisely defined by the conditions either \(W_{\Gamma}(k)\leq 0\) for all \(2\leq k\leq m\) or minimal embedding dimension among those satisfying \(W_{\Gamma}(k)\geq 0\) for all \(2\leq k\leq m\) and \(W_{\Gamma}(2)=0\). This constitutes a complete new point of view and we hope that this argument will shed some light on the attempts to solve both the Wilf conjecture and the Fr\"ogoh\"amosa-conjecture. An instance of this utility is shown by the authors in \cite{eliwilf}.
\medskip

Pieces of information about $\Gamma$ are also encoded in its semimodules. A $\Gamma$-semimodule $\Delta$ is a non-empty subset of $\mathbb{Z}$ that is bounded below and satisfies $\Delta+\Gamma\subseteq \Delta$. A $\Gamma$-semimodule is finitely generated, has a unique system of minimal generators and therefore it possesses invariants such as conductor, Frobenius number, embedding dimension, or delta-invariant. Hence it is possible to consider the Wilf function associated to a $\Gamma$-semimodule $\Delta$, say $W_{\Delta}(k)$. For $k=e(\Delta)$, the number $W_{\Delta}(e(\Delta))$ has been already considered in \cite{almiyano-tetris}. In the current paper, we study some properties of the Wilf function and we use this device to conjecture a bound for the Wilf number of $\Gamma$ in the following manner:

\begin{conj}[Bound conjecture]\label{conj:bound}
There exists a semimodule $\Delta$ minimally generated by $[0,g]$ for a gap $g$ of $\Gamma$, such that $W_{\Delta}(2)\geq -W_{\Gamma}(e(\Gamma))$.
\end{conj}

This viewpoint might bring some knowledge in order to solve the fascinating and involved Wilf conjecture, since we are also able to prove

\begin{theorem}
Bound conjecture \ $\implies$ \ Wilf conjecture.
\end{theorem}

It is certainly remarkable that the Wilf function associated to a $\Gamma$-semimodule is related to the Wilf conjecture. We hope that this new viewpoint may be helpful for the understanding of the Wilf conjecture.
\medskip

On the other hand, it is reasonable to ask for a possible extension of Wilf conjecture in the case of a \(\Gamma\)--semimodule \(\Delta\). In this case, the definition of the Wilf function for a semimodule is key in order to provide a good generalization of Wilf conjecture for semimodules over a numerical semigroup. As our previous work shows \cite{almiyano-tetris} (see also Section \ref{subsec:wilffunction}), the natural generalization, i.e. \(W_\Delta(e(\Delta))\geq 0,\) does not work. In this direction, it is a challenging question to find the minimal \(k\) such that \(W_\Delta(k)\geq 0.\) With the help of our previous work \cite{almiyano-tetris}, we will show that in the particular case of \(\Gamma=\langle\alpha,\beta\rangle\) and \(\Delta\) minimally generated by \(\{0,g\},\) with \(g\in\mathbb{N}\setminus\Gamma,\) we have \(W_\Delta(3)\geq 0\) (see Theorem 4.15). However, this seems to be a very particular case. Thus, in Question \ref{extensionwilfsemimodules} we propose the main problems that can be considered as the generalization of Wilf conjecture to semimodules over a numerical semigroup.

\section{Fundamental facts}\label{sec:basics}

Let $\Gamma$ be a numerical semigroup generated by $a_1,\ldots , a_e$; this fact will be expressed by writing $\Gamma=\langle a_1,\ldots , a_e \rangle$. We will assume that $a_1,\ldots , a_e$ are a \emph{minimal system of generators}. We call $m(\Gamma):=\min (\Gamma\setminus\{0\})$ the \emph{multiplicity} of the semigroup; it is a trivial observation that $m(\Gamma)=a_1$. For generalities on numerical semigroups the reader is referred to the book of Rosales and Garc\'ia-S\'anchez \cite{RosalesGarciaSanchez}. 
\medskip

Over a numerical semigroup $\Gamma$ it is possible to define a module structure in analogy to ring theory: A non-empty subset $\Delta \subseteq \mathbb{Z}$ is said to be a $\Gamma$-semimodule if $\Delta+\Gamma \subseteq \Delta$, where the set $\Delta + \Gamma$ is understood as all possible sums $a+\gamma$ with $a\in \Delta$, $\gamma\in \Gamma$. The set of all $\Gamma$-semimodules with respect to this addition has a structure of additive binoid, with $\Gamma$ as a neutral element, and $\mathbb{N}$ as an absorbent element; notice that, in particular, $\Gamma$ itself is a $\Gamma$-semimodule.
\medskip

A system of generators of $\Delta$ is a subset $\mathcal{E} \subseteq \Delta$ with
$$
\bigcup_{x\in\mathcal{E}} (x+\Gamma) =\Delta.
$$
It is called minimal if no proper subset of $\mathcal{E}$ generates $\Delta$. Every $\Gamma$-semimodule $\Delta$ is finitely generated, and possesses a \emph{minimal} system of generators. Minimal systems of generators of $\Gamma$-semimodules containing $0$ are well-understood: they are of the form $I=[g_0=0,g_1, \ldots, g_r]$ where $|g_i-g_j|$ is a gap of $\Gamma$ for every $i,j\in \{0,\ldots r \} $ with $i\neq j$, see \cite{mu}. Notice that in particular $g_i$ is a gap of $\Gamma$ for every $i\in\{1,\ldots ,r\}$, and the \emph{embedding dimension} $r=:e(\Delta)$ of $\Delta$ is bounded by $0\leq r\leq m(\Gamma)-1$, cf. \cite{mu}. For $\Delta=\Gamma$, the embedding dimension $e(\Gamma)$ of $\Gamma$ is bounded below by $2$ and above by the multiplicity $m(\Gamma)$ of  the semigroup, cf. \cite[Proposition 2.10]{RosalesGarciaSanchez}. If $e(\Gamma)=2$ resp. $e(\Gamma)=m(\Gamma)$ we say that $\Gamma$ has minimal resp.~maximal embedding dimension.
\medskip

The elements in the set $\mathbb{N}\setminus \Delta$, which is finite, are called gaps of $\Delta$. The cardinality $g(\Delta)$ of the set of gaps of $\Delta$ is called the genus of $\Delta$. The maximal gap with respect to the usual total ordering in $\mathbb{Z}$ is called the \emph{Frobenius number} of $\Delta$, written $F(\Delta)$. The number $c(\Delta):=F(\Delta)+1$ is called the conductor of $\Delta$. Moreover, the delta-invariant of $\Gamma$ is defined to be
$$
\delta (\Delta):= |\{x\in \Delta : x < c(\Delta)\}|.
$$
In all these invariants, the dependency of $\Delta$ will be dropped out from the notation whenever no risk of confusion arises.
\medskip

There is a remarkable system of generators ---by no means minimal--- that can be attached to a numerical semigroup $\Gamma$: let $s \in \Gamma\setminus \{0\}$, the Ap\'ery set of $\Gamma$ with respect to $s$ is defined to be the set
\[
\mathrm{Ap}(\Gamma, s)=\{w\in\Gamma\,:\;w-s\notin \Gamma\}.
\]

Observe that the cardinality of $\mathrm{Ap}(\Gamma, s)$ is $s$, and that $\mathrm{Ap}(\Gamma, s)=\{w_0<w_1<\dots<w_{s-1}\}$ where $w_i=\min \{ z\in \Gamma : z \equiv i \ \mathrm{mod}\ s\}$; obviously, $w_0=0$. We will always consider the particular case $s=m:=m(\Gamma)$, for which $w_1=a_2$ and $w_{m-1}=c-1+a_1=c+m-1$. Moreover, the Apéry set with respect to the multiplicity $m$ has the following property.

\begin{lemma}\label{lem:apminimalgens}
Let $\Gamma=\langle a_1,\ldots , a_e \rangle$ be the numerical semigroup minimally generated by \(\{a_1,\dots,a_e\}\). We denote \(a_1:=m\), then 
\[
\{a_2,\dots,a_e\}\subset\mathrm{Ap}(\Gamma, m)=\{w\in\Gamma\,:\;w-m\notin \Gamma\}.
\]
\end{lemma}

\begin{proof}
	If \(a_j\notin\mathrm{Ap}(\Gamma, m)\) then \(a_j=\nu+m\) with \(\nu\in\Gamma\) and \(\nu<a_j\), which contradicts the fact that \(a_j\) is a minimal generator of \(\Gamma\).
\end{proof}

\medskip

A numerical semigroup $\Gamma$ can be endowed with the following partial ordering: for any $s,t\in \Gamma$ we set \(s\preceq t\) if and only if there exists  \(u\in\Gamma\) such that \(s+u=t\). Thus we may define the following two subsets of the Ap\'ery set of $\Gamma$ with respect to the multiplicity $m=m(\Gamma)$: 
\[\mathrm{min}~\mathrm{Ap}(\Gamma,m):=\{w\in\mathrm{Ap}(\Gamma,m)\setminus\{0\}|\;w\,\text{is minimal with respect to}\preceq\},\]
\[\mathrm{max}~\mathrm{Ap}(\Gamma,m):=\{w\in\mathrm{Ap}(\Gamma,m)\setminus\{0\}|\;w\,\text{is maximal with respect to}\preceq\}.\]
The latter set leads to the definition of the \emph{type} of \(\Gamma\) as the cardinality \(t(\Gamma):=|\mathrm{max}~\mathrm{Ap}(\Gamma,m)|\). An important property of the type of a semigroup was proven in \cite[Theorem 22]{fgh}.

\begin{proposition}[Fr\"oberg, Gottlieb, H\"aggkvist]\label{prop:type}
	Let \(\Gamma\) be a numerical semigroup, then \[c(\Gamma)\leq\delta(\Gamma)\cdot(t(\Gamma)+1).\]
\end{proposition}

In particular, any numerical semigroup satisfying \(t(\Gamma)+1\leq e(\Gamma)\) satisfies Wilf conjecture. For example, any numerical semigroup satisfying \(1\leq m(\Gamma)-e(\Gamma)\leq 2\) has \(t(\Gamma)+1\leq e(\Gamma)\) as the proof of \cite[Lemma 6]{samartano} shows.

%


\begin{rem}
	Not every numerical semigroup satisfies the property that \(t(\Gamma)+1\leq e(\Gamma)\), for example \(\Gamma=\langle 213,216,226,227\rangle\) is a numerical semigroup with \(t(\Gamma)=14\) and \(e(\Gamma)=4\), as one can check with GAP \cite{gapdelgado,gap}.
\medskip
	
Observe also that \(\Gamma\) belongs to the family given in \cite{fgh} defined as $\Gamma_{s,n}=\langle s,s+3,s+3n+1,s+3n+2 \rangle$ for $n\geq 2$, $r\geq 3n+2$ and $s=r(3n+2)+3$. However in \cite{fgh} they wrongly claim, as the previous example shows, that this family has type $t(\Gamma_{s,n})=2n+3$. We thank D.I. Stamate for letting us know this mistake, see his Example 3.3. in \cite{stamate}. 
\end{rem}

Given a numerical semigroup $\Gamma$, we will be concerned with the study of the integral interval $[0,c+m]$; of course, this interval is meant to be $[0,c+m]\cap \mathbb{N}$, but we leave out the intersection with $\mathbb{N}$ in order to discharge the notation, and we will assume this and all occurring intervals to be in $\mathbb{N}$. In analogy to \cite{mossam}, we will make a partition of this interval in subintervals of length $m-1$, say
\[
I_\alpha:=[\alpha m,(\alpha+1)m-1], \ \mbox{for} \ \alpha =0, 1, \ldots, \Big \lfloor \frac{c+m}{m} \Big \rfloor.
\]

Let $L:=\lfloor \frac{c-1}{m}\rfloor=\lfloor \frac{w_{m-1}}{m}\rfloor-1$ denote the integer part of the quotient between the conductor of \(\Gamma\) and its multiplicity minus $1$. Hence,  we can write \(c-1=Lm+\rho'\) with \(0\leq \rho'\leq m-1\) and \(\rho'\neq 0\) because \(c-1\neq \Gamma\). Therefore, we can rewrite \(c= Lm+\rho\) with \(\rho=\rho'+1\) and \(2\leq\rho\leq m\).
\medskip

Following the notation of \cite{samartano} and \cite{mossam}, for $j=1, \ldots , m-1$ we define
$$
\eta_j=|\{\alpha\in\mathbb{N} : | I_\alpha\cap\Gamma |=j\}|\quad\text{and}\quad\epsilon_j=|\{\alpha\in\mathbb{N} : | I_\alpha\cap\Gamma |=j,\,0\leq\alpha\leq L-1\}|,
$$
as well as
$$
n_{\alpha}=|\{s\in \Gamma \cap I_{\alpha}: s<F \}|.
$$

The numbers $\eta_j,\epsilon_j$ can be computed from the Ap\'ery set $\mathrm{Ap}(\Gamma, m)$ in the following way:

\begin{lemma}[\cite{samartano}, Proposition 13]\label{lemma:21}
For any $j=1, \ldots , m-1$ we have
	\[
	\eta_j=\Big \lfloor \frac{w_j}{m}\Big \rfloor-\Big \lfloor \frac{w_{j-1}}{m}\Big \rfloor.
	\]
\end{lemma}

Moreover, Sammartano \cite[Proposition 9]{samartano} proved:

\begin{lemma}\label{lemmaS}
The numbers $n_{\alpha}$ satisfy the following properties:
\begin{itemize}
\item[(i)] For $\alpha=0,\ldots , L$, we have that $1\leq n_{\alpha}=|\Gamma\cap I_{\alpha}| \leq m-1$.
\item[(ii)] If $0\leq \alpha <\beta \leq L-1$, then $n_{\alpha} \leq n_{\beta}$.
\item[(iii)] $\delta(\Gamma)=n_0+n_1+\cdots + n_L$.
\end{itemize}
\end{lemma}

The leitmotiv of this paper is to study the ratio \(c(\Gamma)/\delta(\Gamma)\). If \(k\in\mathbb{Z}\), we call \emph{Wilf type inequality} to an inequality of the shape
\[k\delta(\Gamma)\geq c(\Gamma).\]
Observe that if we have to our disposal a Wilf type inequality then the ratio \(c(\Gamma)/\delta(\Gamma)\leq k\). The classification of those \(k\in\mathbb{Z}\) such that \(k\) satisfies a Wilf type inequality is crucial for our purpose. With the previous notation, one can find the following characterization for a \(k\) satisfying a Wilf type inequality.

\begin{proposition}\label{prop:tecnica}
Let $k\in \mathbb{Z}$. Preserving notation as above, a numerical semigroup $\Gamma$ satisfies the inequality $c(\Gamma)\leq k\delta(\Gamma)$ if and only if
\begin{equation}\label{eqn:doble}
\sum_{j=0}^{L} (kn_j-m(\Gamma)) +m(\Gamma)-\rho \geq 0.
\end{equation}
\end{proposition}

\begin{proof}
By using Lemma \ref{lemmaS}, it is easily checked that
\begin{align*}
c(\Gamma)\leq k\delta(\Gamma) \ \Longleftrightarrow  Lm(\Gamma)+\rho \leq k\sum_{j=0}^{L} n_j&\Longleftrightarrow \sum_{j=0}^L m(\Gamma)+\rho -m(\Gamma) \leq k\sum_{j=0}^{L} n_j \\
&\Longleftrightarrow \sum_{j=0}^L (kn_j-m(\Gamma))+m(\Gamma)-\rho \geq 0.
\end{align*}
\end{proof}

Observe that Proposition \ref{prop:tecnica} is an analogue to \cite[Proposition 10]{samartano} but considering an integer $k$ instead of $e(\Gamma)$ and the summation running up to \(L\). In fact, notice that the inequality (\ref{eqn:doble}) can be reformulated as
$$
\sum_{j=0}^{L-1} (kn_j-m(\Gamma)) +(n_Lk-\rho) \geq 0.
$$

As a consequence of Proposition \ref{prop:tecnica} we obtain

\begin{corollary}\label{cor:useful}
A numerical semigroup $\Gamma$ satisfies $c(\Gamma)\leq m(\Gamma) \delta(\Gamma)$ if and only if
$$
\sum_{j=0}^{L-1} (n_j m-m) +n_Lm-\rho = m\sum_{j=0}^{L}(n_j-1)+m-\rho \geq 0.
$$
\end{corollary}

These technical results conclude all elementary properties we need in order to define the main tool of the article, namely the Wilf function of $\Gamma$.

\section{Wilf function of a numerical semigroup}\label{sec:wilffunction}

In this section we are going to introduce a new instrument for the study of the Wilf conjecture. Our approach takes into consideration the study of the behaviour of the map
\[\begin{array}{cccl}
W_{\Gamma}:&\mathbb{N}&\rightarrow&\mathbb{Z}\\
&k&\mapsto &W_{\Gamma}(k):=k\delta(\Gamma)-c(\Gamma).
\end{array}\]
The function $W_{\Gamma}$ will be called the \emph{Wilf function} of the semigroup $\Gamma$. As already mentioned, for $k=e(\Gamma)=e$, the nonnegativity $W_{\Gamma}(e)\geq 0$ expresses thus the statement of Wilf's conjecture; indeed, the Wilf function contributes to the understanding of Wilf's conjecture. For the moment we will investigate this function along the remainder of the current section.
\medskip

First of all, we recall that a numerical semigroup $\Gamma$ is said to be symmetric if for every $z\in \mathbb{Z}$ one has that $z\in \Gamma \ \Longleftrightarrow \ F(\Gamma)-z \notin \Gamma$. This is equivalent to say that $c(\Gamma)=2g(\Gamma)=2\delta(\Gamma)$, see e.g. \cite[Corollary 4.5]{RosalesGarciaSanchez}. In particular, the Wilf function is nonpositive for $k=2$:

\begin{proposition}\label{prop:2}
Let $\Gamma$ be a numerical semigroup, then $W_{\Gamma}(2)\leq 0$, and the equality holds if and only if $\Gamma$ is symmetric.
\end{proposition}

\begin{proof}
Since $c(\Gamma)\leq 2 \delta (\Gamma)$, then the first assertion is clear. The equality holds in virtue of \cite[p.~80, Proposition~7]{serre}.
\end{proof}

A more involved result assures that the Wilf function $W_{\Gamma}(k)$ is nonnegative for $k=m$:

\begin{theorem}\label{thm:m}
Let $\Gamma$ be a numerical semigroup of multiplicity \(m\), then $W_{\Gamma}(m)\geq 0$, and the equality holds if and only if $\Gamma=\langle m, qm+1, \ldots , qm+(m-1)\rangle$ for integers $m,q$ such that $m>1$ and $q>0$.
\end{theorem}

\begin{proof}
The fact that $W_{\Gamma}(m)\geq 0$ follows straightforward from Corollary \ref{cor:useful}. In order to prove the characterization of the equality, the converse is clear: If $\Gamma = \langle m, qm+1, \ldots , qm+(m-1)\rangle$, then $W_{\Gamma}(m)=m\delta-c=mq-qm=0$. So let us prove the direct implication, and assume the existence of a semigroup $\Gamma=\langle a_1=m,a_2,\ldots , a_e \rangle $ with $a_i$ minimal generators such that $W_{\Gamma}(m)=m\delta-c=0$. 
\medskip

Our proof starts with the computation of $\delta$ by means of an extensive use of the statements in Lemma \ref{lemmaS} as well as Lemma \ref{lemma:21}:

\begin{align*}
\delta=&\sum_{j=0}^{L} n_j= \sum_{j=1}^{m-1}(\eta_j\cdot j) +\rho-m =  \sum_{j=1}^{m-1} \bigg (\sum_{i=j}^{m-1} \eta_i \bigg ) +\rho-m \\
=& \sum_{j=1}^{m-1} \bigg ( \sum_{i=j}^{m-1} \Big \lfloor \frac{w_i}{m}\Big \rfloor-\Big \lfloor \frac{w_{i-1}}{m}\Big \rfloor \bigg) + \rho-m =  \sum_{j=1}^{m-1} \Big ( \Big \lfloor \frac{w_{m-1}}{m}\Big \rfloor-\Big \lfloor \frac{w_{j-1}}{m}\Big \rfloor \Big ) +\rho - m \\
= & (m-1) \Big \lfloor \frac{w_{m-1}}{m}\Big \rfloor -\sum_{j=1}^{m-2} \Big \lfloor \frac{w_j}{m}\Big \rfloor+\rho - m\\
= & m \Big \lfloor \frac{w_{m-1}}{m} \Big \rfloor - \sum_{j=0}^{m-1} \Big \lfloor \frac{w_j}{m}\Big \rfloor +\rho - m.
\end{align*}

Now we appeal to the writing $c=Lm+\rho$, with $L=\Big \lfloor \frac{w_{m-1}-m}{m}\Big \rfloor =\Big \lfloor \frac{w_{m-1}}{m}\Big \rfloor -1$ and $2\leq \rho \leq m$.  By the previous equalities we deduce
\begin{align*}
0=m\delta-c=& m \bigg ( m \Big \lfloor \frac{w_{m-1}}{m} \Big \rfloor - \sum_{j=0}^{m-1} \Big \lfloor \frac{w_j}{m}\Big \rfloor +\rho - m  \bigg ) -Lm-\rho\\
=& m\bigg ( m \Big \lfloor \frac{w_{m-1}}{m} \Big \rfloor - \sum_{j=0}^{m-1} \Big \lfloor \frac{w_j}{m}\Big \rfloor +\rho - m \bigg ) -m \bigg ( \Big \lfloor \frac{w_{m-1}}{m} \Big \rfloor -1 \bigg ) -\rho\\
=&m\bigg ( (m-1)\Big \lfloor \frac{w_{m-1}}{m} \Big \rfloor -\sum_{j=0}^{m-1} \Big \lfloor \frac{w_{j}}{m} \Big \rfloor +1+\rho-m  \bigg )-\rho. 
\end{align*}

For the sake of simplicity set $A:=(m-1)\Big \lfloor \frac{w_{m-1}}{m} \Big \rfloor -\displaystyle \sum_{j=0}^{m-1} \Big \lfloor \frac{w_{j}}{m} \Big \rfloor +1+\rho-m$, then the previous reasoning show that $\rho=A \cdot m$, i.e., $\rho$ is a multiple of $m$ that varies in the range $2\leq \rho \leq m$, therefore it must be $A=1$ and $\rho=m$.
\medskip

We are now in a position to show that all Apéry elements have the same integral part if we divide them by the multiplicity of the semigroup. To this aim, we observe that, since $A=1$ and $\rho=m$, we have
$$
(m-1)\Big \lfloor \frac{w_{m-1}}{m} \Big \rfloor -\sum_{j=1}^{m-1} \Big \lfloor \frac{w_{j}}{m} \Big \rfloor =0.
$$
From this,
$$
(m-2)\Big \lfloor \frac{w_{m-1}}{m} \Big \rfloor =\sum_{j=1}^{m-2} \Big \lfloor \frac{w_{j}}{m} \Big \rfloor.
$$
Moreover, since the Apéry set is ordered by \(w_0<w_1<\dots,w_{m-1}\) it is clear that
$$
\sum_{j=1}^{m-2} \Big \lfloor \frac{w_{j}}{m} \Big \rfloor \leq (m-2) \Big \lfloor \frac{w_{m-1}}{m} \Big \rfloor,
$$
since for any $j=1,\ldots , m-1$ it holds that $0\leq \Big \lfloor \frac{w_{j}}{m} \Big \rfloor  \le \Big \lfloor \frac{w_{m-1}}{m} \Big \rfloor$. Altogether we obtain
$$
\Big \lfloor \frac{w_{m-1}}{m} \Big \rfloor =\Big \lfloor \frac{w_{j}}{m} \Big \rfloor  \ \mbox{for any} \ j=1,\ldots , m-1. 
$$

\medskip

We want now to inspect the form of the minimal generators $a_2,\ldots , a_e$. Starting with $a_2$, since this coincides with $w_1$, we have
$$
\Big \lfloor \frac{a_2}{m} \Big \rfloor = \Big \lfloor \frac{w_{m-1}}{m} \Big \rfloor = L+1,
$$
where the last equality holds by the writing $c=Lm+\rho=(L+1)m$ under our proven condition $\rho=m$. Therefore there exists an integer $\alpha_2$ such that
$$
a_2=(L+1)m+\alpha_2 \ \mbox{with} \ 1\leq \alpha_2 \leq m-1.
$$
Moreover, by Lemma \ref{lem:apminimalgens} there exists $j_2\in \{2,\ldots , e-1\}$ such that $a_3=w_{j_2}$ and again there is an integer $\alpha_3$ such that
$$
a_3=(L+1)m+\alpha_3 \ \mbox{with} \ 1\leq \alpha_3 \leq m-1.
$$
Since $a_2<a_3$, it is obvious that $1\leq \alpha_2<\alpha_3$. 
\medskip

An easy reasoning by induction provides the shape of each minimal generator of $\Gamma$, namely
$$
a_i=(L+1)m+\alpha_i \ \mbox{with} \ 1\leq \alpha_i \leq m-1.
$$
for any $i=2,\ldots , e$. It remains thus to prove that $e=m$ and $\alpha_i=i-1$ for $i=2,\ldots , e=m$. 
\medskip

On the contrary, suppose that $e<m$. From what has already been proved, we may write
\begin{equation}\label{eqn:star}
a_2=w_1<a_3=w_{j_2}<a_4=w_{j_3}<\cdots <a_{e}=w_{j_{e-1}}=w_{e-1} <w_e<\cdots < w_{m-1},
\end{equation}
where $\{j_2, j_3,\ldots , j_{e-1}\}=\{2,3,\ldots , e-2\}$. Without loss of generality we may assume that $j_i=i$ for any $i=2,\ldots , e-2$. Therefore $w_e$ is not a minimal generator of $\Gamma$, although it belongs itself to $\Gamma$. This guarantees the existence of nonnegative integers $b_0,b_1,\ldots , b_{e-1}$ such that
\begin{align*}
w_e=& b_0m+b_1w_1+b_2w_2+\cdots + b_{e-1}w_{e-1}\\
=& b_0 m + \sum_{j=1}^{e-1} \Big ( (L+1)m+\alpha_{j+1} \Big ) b_j  \\
=& b_0 m + \sum_{j=1}^{e-1} (L+1)m b_j  + \sum_{j=1}^{e-1} \alpha_{j+1} b_j  \\
=& \Big ( b_0+\sum_{j=1}^{e-1} (L+1)b_j\Big ) m + \sum_{j=1}^{e-1} \alpha_{j+1} b_j.
\end{align*}

Observe that the right-hand side summation is positive, since $b_j\geq 0$ and $\alpha_{j+1} >0$ for any $j=1,\ldots ,e-1$. There must thus exist $i_0\in\{1,\ldots , e-1\}$ such that $b_{i_0}\neq 0$. This allows us the writing
$$
w_e=b_{i_0} (L+1)m+\Big ( b_0 + \sum_{j=1\atop j\ne i_0}^{e-1} b_j (L+1) \Big ) m + \sum_{j=1}^{e-1} \alpha_{j+1} b_j.
$$
On the other hand, we may write
$$
w_e=(L+1)m+\epsilon \ \mbox{with} \ \alpha_{e-1}<\epsilon \leq m-1
$$
(the cases $\epsilon \leq \alpha_{e-1}$ are all excluded in view of equation (\ref{eqn:star})). This implies that $b_j=0$ for any $j\in\{1,\ldots, e-1\}$, $j\neq i_0$, and moreover $b_{i_0}=1$. By the above,
$$
w_e=(L+1)m+\alpha_{i_0}=w_{i_0-1}<w_e,
$$
a contradiction. Therefore $e=m$.
\medskip

Our next and last claim is that $\alpha_i=i-1$ for any $i=2,\ldots , m$; but this is easy, since there are $m-1$ values $\alpha_2<\alpha_3<\cdots < \alpha_m$ satisfying $1\leq \alpha_i \leq m-1$ each of them. 
\medskip

Hence $\Gamma=\langle m,(L+1)m+1,\ldots , (L+1)m+m-1\rangle$ with $m\geq 2$ and $L+1\geq 1$, and the proof is complete.
\end{proof}

\begin{rem}
	In the proof of Theorem \ref{thm:m} we have indeed shown the following:
	$$
	B:=(m-1) \Big \lfloor \frac{w_{m-1}}{m} \Big \rfloor - \sum_{j=0}^{m-1} \Big \lfloor \frac{w_{j}}{m} \Big \rfloor =0 \ \Longleftrightarrow \ \Gamma = \Big \langle m,  \Big \lfloor \frac{w_{m-1}}{m} \Big \rfloor m+1,\ldots ,  \Big \lfloor \frac{w_{m-1}}{m} \Big \rfloor m+m-1 \Big \rangle.
	$$
	In particular, since $B \geq 0$, we have the strict inequality $B>0$ for any numerical semigroup different from the one occurring in the statement of Theorem \ref{thm:m}.
\end{rem}
The annihilation of the Wilf function is related to the type of \(\Gamma\) in the following sense.
\begin{corollary}
	Let \(\Gamma\) be such that \(W_\Gamma(m)=0\) then \(t(\Gamma)=m-1\). In particular, \(t(\Gamma)=m-1\) if $\Gamma=\langle m, qm+1, \ldots , qm+(m-1)\rangle$ for integers $m,q$ such that $m>1$ and $q>0$. \end{corollary}
\begin{proof}
	By linearity of the Wilf function and Proposition \ref{prop:type} we have \(0=W_\Gamma(m)\geq W_\Gamma(t(\Gamma)+1)\geq 0.\) Then \(W_\Gamma(m)=W_\Gamma(t(\Gamma)+1)=0\) and \(t(\Gamma)=m-1.\) The second part follows straightforward from Theorem \ref{thm:m}, since \(W_\Gamma(m)=0\) is equivalent to $\Gamma=\langle m, qm+1, \ldots , qm+(m-1)\rangle$.

\end{proof}
The following Figure \ref{Fig:wilf} sketches the behaviour of the Wilf function:

\begin{figure}[H]
\begin{tikzpicture}[line cap=round,line join=round,>=triangle 45,x=0.35cm,y=0.35cm]

\clip(-0.5,-6.) rectangle (15.5,8.);
\draw [line width=1.pt] (0.,-4.)-- (13.,6.);
\draw [->,line width=1.pt] (0.,-5.) -- (0.,7.42);
\draw [->,line width=1.pt] (0.,0.) -- (15.,0.);
\begin{scriptsize}
\draw [fill=uuuuuu] (0.,-4.) circle (2.5pt);
\draw[color=black] (3.0,-3.91) node {$W_\Gamma(0)=-c(\Gamma)$};
\draw [fill=uuuuuu] (13.,6.) circle (2.5pt);
\draw[color=black] (13.1,4.8) node {$W_{\Gamma}(m)$};
\draw[color=black] (8.36,1.2) node {$W_{\Gamma}(k)$};
\draw [fill=uuuuuu] (5.2,0.) circle (2.0pt);
\draw[color=black] (5.,0.71) node {$c/\delta$};
\end{scriptsize}
\end{tikzpicture}
\caption{Sketch of the graph of the function $W_{\Gamma}(k)$.}\label{Fig:wilf}
\end{figure}
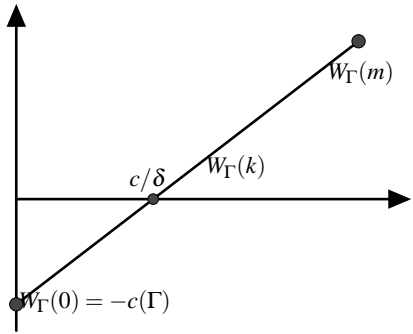

In view of Figure \ref{Fig:wilf}, Theorem \ref{thm:m} together with Proposition \ref{prop:2} suggest that the interesting range through which we must let $k$ run is precisely $2\leq k \leq m$; the case $k=1$ leads to $\Gamma=\mathbb{N}$ and is therefore irrelevant for our purpose. Wilf's question concerns the value $k=e$ that belongs certainly to this range. But in general, $k=e$ is not the minimal value making $W_{\Gamma}(k)$ nonnegative. This means, the Wilf number $W_{\Gamma}(e)$ does not yield in general a sharp bound for the positivity of the Wilf function. From this point of view, it would be certainly interesting to investigate the constant
$$
\mu_{\Gamma}:=\min \{k\in \mathbb{N}: W_{\Gamma}(k)\geq 0\},
$$
where obviously $2\leq \mu_{\Gamma}\leq m$. In other words, as emphasized by Figure \ref{Fig:wilf}, the function $W_{\Gamma}$ considered as a real function has a root between $2$ and $m$, and we propose to find the minimal integer value over the root so that the function is positive. The study of \(\mu_\Gamma\) presents some stimulating challenges. First, let us look at some examples done with the aid of GAP \cite{gapdelgado,gap}.

\begin{ex}\label{ex:fromentineliahou}
Before presenting some computations, we establish the following standard notation: write \(S=\langle x_1\dots,x_s\rangle_r\) for the minimal semigroup that contains \(\{x_1\dots,x_s\}\) and all the integers greater than or equal to \(r\). This notation is widely used e.g. by Delgado in \cite{delgado1}.
\medskip

Consider the numerical semigroups
\begin{align*}
S_1=& \langle 162,1114,1115 \rangle_{9879}& S_2=& \langle 222, 1532, 1533  \rangle_{16647}& S_3=& \langle 172, 327, 328 \rangle_{3437} \\
 S_4=& \langle 88, 100, 102 \rangle_{566} &
S_5=& \langle 88,100,343,345,346,&351,&361,679,680,681,687,693\rangle_{700}
\end{align*}

For each $i=1,\ldots, 5$, we present Table \ref{tab:table1} the embedding dimension $e_i=e (S_i)$, the delta-invariant $\delta_i=\delta (S_i)$, the conductor $c_i=c(S_i)$, the Wilf number $W_i(e_i)=W_{S_i}(e_i)$ as well as $\mu_i=\mu_{S_i}$, the value $W_i(\mu_i)$ and the difference $\Delta_i:=e_i-\mu_i$ for the semigroup $S_i$:

\begin{table}[H]
  \begin{center}
    \begin{tabular}{c|ccccccc} 
      $i$ & $\delta_i$ & $c_i$  & $e_i$ & $\mu_i$ & $\Delta_i$ & $W_i(e_i)$ & $W_i(\mu_i)$   \\
      \hline\hline
      1 & $1109$ & $9879$ & $110$ & \textbf{9} & $101$ & $112111 $ & $102$ \\
      2 & $1935$ & $16647$ & $147$ & \textbf{9} & $138$ & $267798$ & $768$ \\
      3 & $505$ & $3437$ & $97$ & \textbf{7} & $90$ & $45548$ & $98$ \\
      4 & $63$ & $566$ & $63$ & \textbf{9} & $54$ & $3403$ & $1$\\
      5 & $100$ & $700$ & $51$ & \textbf{7} & $44$ & $5100$ & $0$
    \end{tabular}
    \vspace{0.3cm}
      \caption{Invariants of some semigroups.}  \label{tab:table1}
  \end{center}
\end{table}

These examples show how far is in some cases the bound given by the Wilf conjecture (i.e. the bound associated with the embedding dimension) to be sharp. They give also evidences in the direction of the convenience of the constant $\mu_{\Gamma}$ we have just introduced.
\end{ex}

In addition, a companion question to that of the investigation of the constant $\mu_{\Gamma}$ is the following:
\begin{question}\label{ques:mutight}
	To ask for properties of and to characterize those numerical semigroups $\Gamma$ with fixed $\mu_{\Gamma}$. In particular, to characterize those semigroups with \(\mu_\Gamma=e(\Gamma)\).
\end{question}
  The value $\mu_{\Gamma}=2$ is already known: it is equivalent to the case of $\Gamma$ symmetric, which in purely algebraic algebraic means the Gorenstein property. We wonder whether the constant $\mu_{\Gamma}$ encloses some other stimulating feature.
\medskip

To conclude this section, observe that we can use the properties of the type $t(\Gamma)$ of a semigroup $\Gamma$ to give an upper bound for the invariant \(\mu_\Gamma\). 

\begin{proposition}
	Let \(\Gamma\) be a numerical semigroup. Then, \(\mu_\Gamma\leq t(\Gamma)+1.\)
\end{proposition}

\begin{proof}
	Proposition \ref{prop:type} can be stated as \(W_\Gamma(t(\Gamma)+1)\geq 0\) for any numerical semigroup. In particular this means that \(\mu_\Gamma\leq t(\Gamma)+1.\)
\end{proof}




As a general conclusion of this section, we would like to emphasize the importance of the constant \(\mu_\Gamma\). Observe that in most of the previous cases \(\mu_\Gamma\neq e(\Gamma)\). 
\medskip

Moreover, the Wilf function is conceptually very useful: it allows us to guess an explanation of the exceptional numerical semigroups appearing in Fr\"ogoh\"amosa-conjecture, i.e. those with \(W_\Gamma(e)=0\). Observe that the numerical semigroups appearing as candidates for being the only ones with \(W_\Gamma(e)=0\) in Fr\"ogoh\"amosa-conjecture are precisely those for which the Wilf function has a extreme behaviour. 
\medskip

Obviously, since \(W_\Gamma(0)\) and \(W_\Gamma(1)\) are negative for every numerical semigroup \(\Gamma\neq \mathbb{N}\), the first integer where the function can be positive is \(2\); as we have previously shown, \(W_\Gamma(m)\geq 0\) for any numerical semigroup. Therefore, to say that \(2\) and the multiplicity of \(\Gamma\) are extreme values for \(W_{\Gamma}(k)\) means the following: 

\begin{theorem}\label{thm:extreme}
	Let \(\Gamma\) be a numerical semigroup. Then,
	\begin{enumerate}
		\item \(\Gamma=\mathbb{N}\) if and only if \(W_\Gamma(k)\geq 0\) for \(1\leq k\leq m\).
		\item \(\Gamma=\langle m, qm+1, \ldots , qm+(m-1)\rangle\) for \(q\geq 1\) if and only if \(W_{\Gamma}(k)\leq 0\) for all \(1\leq k\leq m\).
		\item \(\Gamma=\langle a,b \rangle\) with \(\gcd(a,b)=1\) is the numerical semigroup with minimal embedding dimension between those satisfying \(W_{\Gamma}(k)\geq 0\) for all \(2\leq k\leq m\) and \(W_{\Gamma}(2)=0\).
	\end{enumerate}
\end{theorem}
\begin{proof}
	The statement \((1)\) is clear.  The second assertion is equivalent to Theorem \ref{thm:m} and the fact that the Wilf function is strictly increasing. Finally, for the assertion \((3)\) let us assume that \(W_{\Gamma}(k)\geq 0\) for all \(2\leq k\leq m\) and \(W_{\Gamma}(2)=0\). By Proposition \ref{prop:2} the numerical semigroups satisfying those conditions are exactly the family of symmetric numerical semigroups. In particular, the symmetric numerical semigroups with minimal embedding dimension  are those of type \(\Gamma=\langle a,b \rangle\) with \(\gcd(a,b)=1\). 
\end{proof}

In Figure \ref{Fig:wilfemb2} and Figure \ref{Fig:wilffrogo2} we try to illustrate this phenomenon for both semigroups $\Gamma=\langle \alpha, \beta\rangle$ of embedding dimension two and semigroups of the form $
\Gamma=\langle m, qm+1, \ldots , qm+(m-1)\rangle.
$

\begin{multicols}{2}
	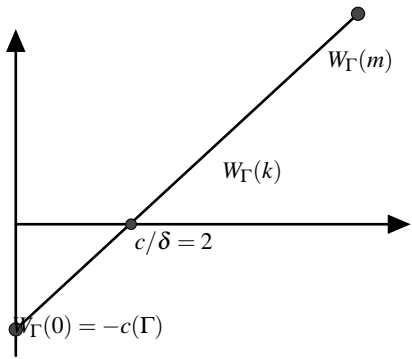
\begin{figure}[H]
		\begin{tikzpicture}[line cap=round,line join=round,>=triangle 45,x=0.35cm,y=0.35cm]
			
			\clip(-0.5,-6.) rectangle (15.5,9.);
			\draw [line width=1.pt] (0.,-4.)-- (13.,8.);
			\draw [->,line width=1.pt] (0.,-5.) -- (0.,7.42);
			\draw [->,line width=1.pt] (0.,0.) -- (15.,0.);
			\begin{scriptsize}
				\draw [fill=uuuuuu] (0.,-4.) circle (2.5pt);
				\draw[color=black] (2.75,-3.91) node {$W_\Gamma(0)=-c(\Gamma)$};
				\draw [fill=uuuuuu] (13.,8) circle (2.5pt);
				\draw[color=black] (13.1,6.2) node {$W_{\Gamma}(m)$};
				\draw[color=black] (9.0,2.0) node {$W_{\Gamma}(k)$};
				\draw [fill=uuuuuu] (4.38,0.) circle (2.0pt);
				\draw[color=black] (6.,-0.71) node {$c/\delta=2$};
			\end{scriptsize}
		\end{tikzpicture}
		\caption{Sketch of the graph of the function $W_{\Gamma}(k)$ with \(e=2\).}\label{Fig:wilfemb2}
	\end{figure}
	\begin{figure}[H]
		\begin{tikzpicture}[line cap=round,line join=round,>=triangle 45,x=0.35cm,y=0.35cm]
			
			\clip(-0.5,-6.) rectangle (15.5,8.);
			\draw [line width=1.pt] (0.,-4.)-- (12.,0);
			\draw [->,line width=1.pt] (0.,-5.) -- (0.,7.42);
			\draw [->,line width=1.pt] (0.,0.) -- (15.,0.);
			\begin{scriptsize}
				\draw [fill=uuuuuu] (0.,-4.) circle (2.5pt);
				\draw[color=black] (5.0,-3.91) node {$W_\Gamma(0)=-c(\Gamma)$};
				\draw [fill=uuuuuu] (12.,0) circle (2.5pt);
				\draw[color=black] (12.1,1) node {$W_{\Gamma}(m)=0$};
				\draw[color=black] (5.5,-1.2) node {$W_{\Gamma}(k)$};
				
			\end{scriptsize}
		\end{tikzpicture}
		\caption{Graph of the function $W_{\Gamma}(k)$ for \(\tiny{\Gamma=\langle m, qm+1, \ldots , qm+(m-1)\rangle}\)}\label{Fig:wilffrogo2}
	\end{figure}
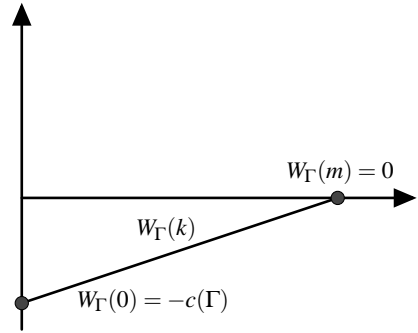
\end{multicols}
 
We finish the section with the observation that the vanishing of the Wilf function imposes a very strong condition on $\Gamma$:

\begin{proposition}\label{prop:dk}
Let $\Gamma$ be a numerical semigroup. If $W_{\Gamma}(k)=0$ for some $k$, then $k\delta \leq (L+1)m$, hence $k\leq m$.
\end{proposition}

\begin{proof}
The first statement is clear by the above reasonings. Moreover, by the statement (iii) in Lemma \ref{lemmaS} we have $\displaystyle \delta = \sum_{j=0}^{L} n_j \geq L+1$, it follows that $(L+1)k\leq \delta k \leq (L+1)m$ and so $k\leq m$. 
\end{proof}
The conclusion $k\leq m$ in Proposition \ref{prop:dk} follows also from the application of the Darboux property to the Wilf function, as it is easily deduced by an inspection of Figure \ref{Fig:wilf}.

\section{Wilf number of a \(\Gamma\)--semimodule}\label{sec:wilfito}

The notion of \emph{Wilf number}  was already introduced by the authors in \cite[Definition 3.1]{almiyano-tetris}. Following the arguments of this paper, let us give the following definition.

\begin{defin}\label{def:wilf number semimodule}
Let $\Delta$ be a $\Gamma$-semimodule, then the Wilf number of $\Delta$ is defined to be
$$
W(\Delta)=e(\Delta)\cdot \delta (\Delta)-c(\Delta).
$$
\end{defin}	
	
Observe that in the particular case $\Delta=\Gamma$, we have \(W(\Gamma)=W_{\Gamma}(e(\Gamma))\) and Wilf's conjecture claims that $W(\Gamma)\geq 0$. 
\medskip

A particular case of a $\Gamma$-semigroup containing $0$ is that minimally generated by $0$ and a gap $g$ of $\Gamma$. Following the notation used in \cite{mu} or \cite{mu2}, if we write $I=[0,g]$ for this minimal system of generators, and $\Delta_I$ for the $\Gamma$-semimodule minimally generated by $I$, then we define the \emph{Wilf number of a gap} $g\in \mathbb{N}\setminus\Gamma$ by assigning the Wilf number of $\Delta_I$, namely
$$
W(g):=W(\Delta_I)=2\delta(\Delta_{I})-c(\Delta_{I}).
$$

In the sequel we will investigate the properties of the Wilf number associated to a gap. It turns out that there are some evidences that might let think of a relation between this function and the Wilf conjecture; this will be explained in Subsect.~\ref{subsec:wilfgap}. On the other hand, in the Wilf number of a $\Gamma$-semimodule $\Delta$ we may replace ---as we did for the semigroup--- the embedding dimension $e(\Delta)$ by an integer $k$, pursuing the idea presented in Sect.~\ref{sec:wilffunction} for the Wilf number of a semigroup. Subsect.~\ref{subsec:wilffunction} is thus devoted to the study of the Wilf function attached to a $\Gamma$-semimodule.

\subsection{Wilf number of a gap}\label{subsec:wilfgap}

As pointed out in a previous paper by the authors \cite{almiyano-tetris}, the Wilf number $W(g)$ associated to a gap $g$ of $\Gamma$ seems to show intrinsic properties of the semigroup itself. Based on this idea, we wonder whether the Wilf number of a gap would help to give an answer to Conjecture \ref{Wilfconjecture}. The manuscript \cite{almiyano-tetris} shows indeed some evidences so that $W(g)$ can take both positive and negative values; but this is bounded as follows:

\begin{theorem}\label{thm:boundwilfgap}
Let $\Gamma$ be a numerical semigroup, and let $g$ be a gap of $\Gamma$, then
$$
\max(W(g))\leq W_{\Gamma}(4).
$$
\end{theorem}

\begin{proof}
Let $I=[0,g]$ be the minimal system of generators of the $\Gamma$-semimodule $\Delta_I$.
Since $\delta(\Delta_{I})\leq\delta(\Gamma)+\delta(\Gamma)-(c(\Gamma)-c(\Delta_{I}))$, we have
\begin{align*}
W(g) = & 2\delta (\Delta_I)-c(\Delta_I) \\
\leq & 4\delta (\Gamma) - 2(c(\Gamma)-c(\Delta_I)) - c(\Gamma) + (c(\Gamma)-c(\Delta_I)) \\
=& 4\delta (\Gamma) - c(\Gamma) - (c(\Gamma)-c(\Delta_I)).
\end{align*}
As $c(\Gamma)-c(\Delta_{I})\geq 0$, we conclude that
$$
W(g) \leq 4\delta(\Gamma) - c(\Gamma) = W_{\Gamma}(4),
$$
for any gap $g$ of $\Gamma$, which proves the theorem.
\end{proof}

In addition, not only the maximum but the whole range of possible values for the Wilf number of a gap is bounded:

\begin{proposition}\label{prop:rangewilfito}
Let $g$ be a gap of a numerical semigroup $\Gamma$, then
$$
\max(W(g))-\min(W(g))< 2\delta(\Gamma).
$$
Furthermore, if $\Gamma$ is symmetric, then the range $\max(W(g))-\min(W(g))$ is strictly bounded above by the conductor $c(\Gamma)$ of $\Gamma$.
\end{proposition}

\begin{proof}
Let $I=[0,g]$ be the minimal system of generators of the $\Gamma$-semimodule $\Delta_I$. We know that
$$
	\max(W(g))\leq 4\delta(\Gamma)-c(\Gamma)-(c(\Gamma)-c(\Delta_{I})),
$$
cf.~proof of Theorem \ref{thm:boundwilfgap}. Moreover
	\[
	\min(W(g))\geq 2\delta(\Gamma)+2-c(\Gamma)-(c(\Gamma)-c(\Delta_{I})),
	\]
hence a straightforward computation shows that
	\[
	\max(W(g))-\min(W(g))\leq 2\delta(\Gamma)-2 < 2\delta(\Gamma).
	\]
If $\Gamma$ is symmetric, then $c(\Gamma)=2\delta (\Gamma)$ (see e.g. \cite[Corollary 4.5]{RosalesGarciaSanchez}), and the second claim follows.
\end{proof}

Proposition \ref{prop:rangewilfito} has a first immediate corollary for the Wilf number of a gap in the case \(\Gamma=\langle \alpha,\beta \rangle.\) For such a semigroup, in our previous work \cite{almiyano-tetris} we proved that \(W(g)\) satisfies certain symmetries of its lattice representation. Moreover, we were able to provide an explicit formula for the Wilf number of a gap, as we will recall in a moment.
\medskip

First of all, the gaps of $\langle \alpha, \beta \rangle$ are also easy to describe: they admit a unique representation $\alpha \beta -a\alpha -b \beta$, where $a\in \ ]0,\beta-1]\cap \mathbb{N}$ and $b\in \ ]0,\alpha-1]\cap \mathbb{N}$. This writing yields a map from the set of gaps of $\langle \alpha, \beta \rangle$ to $\mathbb{N}^2$ given by $\alpha \beta -a\alpha -b \beta \mapsto (a,b)$, which allows us to identify a gap with a point in the lattice $\mathcal{L}=\mathbb{N}^2$; since the gaps are positive numbers, the point lies inside the triangle with vertices $(0,0),(0,\alpha),(\beta, 0)$. From this lattice representation, we were able to provide the following expression for the Wilf number of a gap:

\begin{proposition}\cite[Proposition~3.3]{almiyano-tetris}\label{prop:wilfgap-tetris}
  Let \(g=\alpha\beta-a\alpha-b\beta\) be a gap of $\Gamma$. Then
	\[-W(g)=\Big\{\begin{array}{cc}
	a\alpha-2ab&\text{if}\;\;\;\mathrm{min}\{\Gamma\cap(\Gamma+g)\}=\alpha\beta-b\beta,\\
	b\beta-2ab&\text{if}\;\;\;\mathrm{min}\{\Gamma\cap(\Gamma+g)\}=\alpha\beta-a\alpha.
	\end{array}\]
\end{proposition}

For a gap $g=\alpha\beta-a\alpha-b\beta$ we may also write $W(a,b)$ instead of $W(g)$. We observe:

\begin{lemma}\cite[Lemmas 4.1 and 4.2]{almiyano-tetris}\label{lem:sym} Let \(\Gamma=\langle\alpha,\beta\rangle\) be a numerical semigroup with two generators. If \((a,b)\in\mathcal{L}\) is an integral point inside the triangle delimited by the \(y\)--axis, the line \(y=\lfloor\frac{\alpha}{2}\rfloor\) and the diagonal \(\alpha\beta=x\alpha+y\beta\) then
	\[
	-W(a,b)=W(a,\alpha-b).
	\]
	If \((a,b)\in\mathcal{L}\) is an integral point inside the triangle delimited by the \(x\)--axis, the line \(x=\lfloor\frac{\beta}{2}\rfloor\) and the diagonal \(\alpha\beta=x\alpha+y\beta\) then
	\[
	-W(a,b)=W(\beta-a,b).
	\]
\end{lemma}
In particular, the set of fixed points of each of the previous symmetries is exactly the set of points with \(W(a,b)=0\). The gaps on those different triangles, their respective symmetric ones and the gaps with Wilf number equal to \(0\) constitute a disjoint partition of the set of gaps. Moreover, those symmetries of the Wilf number allow us to define the concepts of supersymmetric and self-symmetric gaps (see \cite[Definition~4.4]{almiyano-tetris}). Even more, one can see that the whole semigroup \(\Gamma=\langle\alpha,\beta\rangle\) can be recovered from them \cite[Theorem~4.5]{almiyano-tetris}.
\medskip

We are now in the conditions to prove the following consequence of Proposition \ref{prop:rangewilfito}. 
\begin{corollary}\label{cor:boundmin}
	Let \(\Gamma=\langle \alpha,\beta \rangle\) be a numerical semigroup and let \(g\in\mathbb{N}\setminus\Gamma\) be a gap. Then
	\[\max(W(g))=-\min(W(g))<\delta(\Gamma).\]
\end{corollary}
\begin{proof}
	Since the Wilf number of a gap satisfies the symmetries of Lemma \ref{lem:sym}, then it is obvious that \(\max(W(g))=-\min(W(g))\). Therefore, Proposition \ref{prop:rangewilfito} reads
	$$
	\max(W(g))-\min(W(g))=2\max(W(g))< 2\delta(\Gamma).
	$$
\end{proof}

Unfortunately, there is no analogous formula to the one of Proposition \ref{prop:wilfgap-tetris} for a semigroup \(\Gamma\) with any number of generators. However, given a semigroup \(\Gamma\) with any number of generators, one can try to use the properties of the Wilf number of any gap to deduce as much as possible information leading to the solution of the Wilf conjecture. To do so, recall that \(W_\Gamma(k)\) is an increasing function of \(k\). Therefore, the following corollary is immediate:
\begin{corollary}
	Let \(\Gamma\) be a numerical semigroup with \(e(\Gamma)\geq 4\). If there exists \(g\in\mathbb{N}\setminus\Gamma\) such that \(W(g)\geq 0\), then \(W_\Gamma(k)\geq 0\) for any \(k\geq 4\).
\medskip	
	In particular, \(\Gamma\) satisfies Wilf conjecture \ref{Wilfconjecture}.
\end{corollary}
\begin{proof}
If there exists a gap \(g\in\mathbb{N}\setminus\Gamma\) such that \(W(g)\geq 0\), then \(\max(W(g))\geq 0\). Therefore, by Theorem \ref{thm:boundwilfgap} we have that \(0\leq \max(W(g))\leq W_{\Gamma}(4)\). Thus, \(W_\Gamma(k)\geq 0\) for any \(k\geq 4\) since the Wilf function is increasing in \(k\).
\end{proof}

However, the existence of a positive value of the Wilf number of a gap is not guaranteed. There are semigroups for which $W(g)<0$ for all the gaps: to find numerical semigroups with \(W(g)<0\) for all the gaps it is enough, by Theorem \ref{thm:boundwilfgap}, to find a numerical semigroup \(\Gamma\) with \(\mu_{\Gamma}>4\). Then, all the semigroups of Example \ref{ex:fromentineliahou} provide examples of numerical semigroups with \(W(g)<0\) for all the gaps. This motivates the following
\begin{question}
	Is there always a gap \(g\) with \(W(g)\geq 0\) for any numerical semigroup with \(\mu_\Gamma\leq 4\)?
\end{question}
\medskip

Despite the inconvenience of having semigroups for which $W(g)<0$ for all the gaps , numerical examples lead us to propose the following conjecture for a lower bound of the Wilf number of a gap:

\begin{conj}[Bound conjecture]\label{conj:bound}
For any semigroup \(\Gamma\), the inequality
$$\min(W(g))\geq -W_{\Gamma}(e(\Gamma))$$
holds.
\end{conj}

Here it is important to notice that a positive answer to Conjecture \ref{conj:bound} would solve the Wilf conjecture:

\begin{theorem}\label{thm:wilfitowins}
The Bound conjecture \ref{conj:bound} implies the Wilf conjecture.
\end{theorem}

\begin{proof}
Since the Wilf conjecture is known to be true for semigroups with \(e(\Gamma)<4\) \cite{dm}, let us assume \(e(\Gamma)\geq 4\). By Theorem \(\ref{thm:boundwilfgap}\) we have \(\max(W(g))\leq W_{\Gamma}(4)\leq W_{\Gamma}(e(\Gamma))\), since \(W_\Gamma(k)\) is increasing in \(k\). Therefore, if the Bound conjecture \ref{conj:bound} is true, then
\[W_{\Gamma}(e)\geq\max(W(g))\geq\min(W(g))\geq -W_{\Gamma}(e(\Gamma)).\]
Hence \(|W(g)|\leq W_{\Gamma}(e(\Gamma))\) which implies \(W_{\Gamma}(e(\Gamma))\geq 0\).
\end{proof}

\begin{rem}
	Moreover, given a semigroup $\Gamma$ with embedding dimension $e$, for the proof of the Bound conjecture \ref{conj:bound} it would be enough to prove the existence of a gap $g$ of $\Gamma$ and of an integer \(k\) such that $k\leq e-2$ and \(W(g)\geq -W_{\Gamma}(k)\). In such a case, one would have
	$$
	\min(W(g))\geq \max(W(g))-2\delta(\Gamma)\geq-W_{\Gamma}(k)-2\delta(\Gamma)\geq -W_{\Gamma}(e),
	$$
	which implies the claim.
\end{rem}

In view of Theorem \ref{thm:wilfitowins}, it might be useful to understand the behaviour of the Wilf function for an arbitrary $\Gamma$-semimodule. This study is precisely the content of the next subsection.

\subsection{Wilf function of a $\Gamma$-semimodule with an arbitrary number of minimal generators}\label{subsec:wilffunction}

At this point a natural definition comes up to our mind. Let $\Delta$ be a $\Gamma$-semimodule  of embedding dimension $e(\Delta)$. We define the \emph{Wilf function of a $\Gamma$-semimodule} as the function $W_{\Delta}:\mathbb{N}\to \mathbb{Z}$ given by
$$
W_{\Delta}(k)=k\delta(\Delta)-c(\Delta).
$$
Observe that this is the natural generalization of the Wilf function of a semigroup. From this point of view, a generalization of the Wilf conjecture would seem to be natural. However, for $k=e(\Delta)$ this function admits in general a chaotic behaviour. Look at the following example.

\begin{ex}
For $\Gamma=\langle 6,8,35 \rangle$, let $\Delta_{[0,g]}$ be the $\Gamma$-semimodule minimally generated by $[0,g]$ with $g$ a gap of $\Gamma$. Table \ref{tab:table2} presents the values taken by the Wilf function $W_{\Delta_I}(2)$:
\begin{table}[h!]
  \begin{center}
    \begin{tabular}{cc||cc||cc} 
      $g$ &$W_{\Delta_I}(2)$ &   $g$ &$W_{\Delta_I}(2)$ &   $g$ &$W_{\Delta_I}(2)$   \\
      \hline\hline
      1 & $0$ & 11 & $0$ & 27 & $0$ \\
      2 & $-2$ & 13 & $0$ & 29 & $0$   \\
      3 & $0$ & 15 & $0$ & 31 & $4$   \\
      4 & $0$ & 17 & $0$ & 33 & $2$  \\
      5 & $0$& 19 & $0$ & 37 & $2$ \\
       7 & $0$& 21 & $0$ & 39 & $4$  \\
      9 & $0$& 23 & $0$ & 45 & $4$  \\
      10 & $2$& 25 & $4$ &  &   \\
    \end{tabular}
    \vspace{0.3cm}
      \caption{Gaps and their Wilf numbers for $\Gamma=\langle 6,8,35 \rangle$.}  \label{tab:table2}
  \end{center}
\end{table}
\end{ex}

Our previous work \cite{almiyano-tetris} shows that, even for a semigroup \(\Gamma=\langle \alpha, \beta \rangle\), the Wilf function of \(\Delta=\Delta_{[0,g]}\) at the embedding dimension can take both positive and negatives values. Therefore, the behaviour of the Wilf function at \(k=\mathrm{e}(\Delta)\) is bad even for the \textquotedblleft easy case\textquotedblright~ \(\mathrm{e}(\Delta)=2\). Taking this into account, we have several questions:

\begin{question}\label{extensionwilfsemimodules}
	Let \(\Gamma\) be a numerical semigroup and \(\Delta\) be a \(\Gamma\)--semimodule.
	\begin{itemize}
		\item [(1)] Find a  characterization of $$
		\mu_{\Gamma,\Delta}:=\min \{k\in \mathbb{N}: W_{\Delta}(k)\geq 0\;\text{for all $\Gamma$-semimodules }  \Delta\}.
		$$
		\item [(2)] Is \(\mu_{\Gamma,\Delta}\) related to any invariant of \(\Gamma\)?
		\item [(3)] For those \(\Gamma\)-semimodules $\Delta$ with \(\mathrm{e}(\Delta)\) minimal generators, characterize
		$$
		\mu_{\Delta,\mathrm{e}(\Delta)}:=\min \{k\in \mathbb{N}: W_{\Delta}(k)\geq 0\;\text{for all } \Delta\;\text{with \(\mathrm{e}(\Delta)\) minimal generators}\}.
		$$
		\item [(4)] Can \(\mu_{\Delta,\mathrm{e}(\Delta)}\) be computed from \(\mu_{\Gamma,\Delta}\)? 
		\item [(5)] For a fixed \(\Gamma\), describe those \(\Gamma\)--semimodules \(\Delta\) such that \(\mu_{\Delta,\mathrm{e}(\Delta)}\leq\mathrm{e}(\Delta)\).
		\item [(6)] Find a characterization of those numerical semigroups \(\Gamma\) such that for any \(\Gamma\)--semimodule \(\Delta\) one has \(\mu_{\Delta,\mathrm{ed}(\Delta)}\leq\mathrm{e}(\Delta)\), i.e. \(W_{\Delta}(\mathrm{e}(\Delta))\geq 0\).
	\end{itemize}
\end{question}

These questions arise naturally after the whole discussion of this paper. Our manuscript \cite{almiyano-tetris} allows us to work with an explicit expression (see Proposition \ref{prop:wilfgap-tetris}) of the Wilf number of a gap in the case \(\Gamma=\langle \alpha,\beta \rangle\). Thus, we are able to provide a partial solution to Question \ref{extensionwilfsemimodules}~(3) in the case \(\Gamma=\langle \alpha,\beta \rangle\). So let us compute a bound for $\mu_{\Delta,\mathrm{ed}(\Delta)}$ in Question \ref{extensionwilfsemimodules}~(3) in the case of a \(\Gamma\)--semimodule \(\Delta\) with \(\mathrm{ed}(\Delta)=2\). 

\begin{theorem}\label{thm:final}
	Let \(\Gamma=\langle\alpha,\beta\rangle\) be a numerical semigroup with two generators. Let \(g\in\mathbb{N}\setminus\Gamma\) be a gap of the semigroup. Let \(\Delta=[0,g]\) be the \(\Gamma\)--semimodule associated to \(g\). Then, 
	\[W_{\Delta}(3)\geq 0.\]
	In particular, \(\mu_{\Delta,2}= 3\).
\end{theorem}
\begin{proof}
	 Let \(\Delta=[0,g]\) be the \(\Gamma\)--semimodule associated to a gap \(g=\alpha\beta-a\alpha-b\beta\).
	 First of all, let us denote by \(\mathcal{T}_{r}\) the set of points of \(\mathcal{L}\) inside the triangle delimited by the \(x\)--axis, the line \(x=\lfloor\frac{\beta}{2}\rfloor\) and the diagonal \(\alpha\beta=x\alpha+y\beta\) and  \(\mathcal{T}_{u}\) the set of points of \(\mathcal{L}\) inside the triangle delimited by the \(y\)--axis, the line \(y=\lfloor\frac{\alpha}{2}\rfloor\) and the diagonal \(\alpha\beta=x\alpha+y\beta\).
\medskip
	 
	 Recall that, by \cite[Proof of Theorem 4.5]{almiyano-tetris} (see also Lemma \ref{lem:sym}), the values of \(W_{\Delta}(2)\) are
	 \[\{0\}\cup\{\pm W(g):\;g\in\mathcal{T}_r\cup\mathcal{T}_{u}\}.\]
	 Moreover, by \cite[Lemma 4.1]{almiyano-tetris}, \cite[Proposition 4.8]{almiyano-tetris} and Proposition \ref{prop:wilfgap-tetris},
	 \begin{enumerate}
	 	\item If \(g\in\mathcal{T}_{u}\) then \(-W(g)=a\alpha-2ab<0\).
	 	\item If \(g\in\mathcal{T}_{r}\) then \(-W(g)=b\beta-2ab<0\).
	 \end{enumerate}
 	Now, let us show that \(-W(g)\leq\delta(\Delta)\) for any \(g\in\mathbb{N}\setminus\Gamma\). Since \(\delta(\Delta)>0\) it is enough to prove it for those \(g\notin\mathcal{T}_{u}\cup\mathcal{T}_{r}\). 
\medskip
 	
 	First of all, we have the following expression for \(\delta(\Delta_{[0,g]})\)  given by \cite[Proposition 3.2]{almiyano-tetris}:
 	$$
 	\delta(\Delta_{[0,g]})=c(\Delta_{[0,g]})-\delta(\Gamma)+ab.
 	$$
 	
 	Let us consider \(g\notin\mathcal{T}_{u}\cup\mathcal{T}_{r}\), then \(g=\alpha\beta-a\alpha-b\beta\) with \(1\leq a\leq\beta/2\) and \(1\leq b\leq \alpha/2\). Let us assume that \(\mathrm{min}\{\Gamma\cap(\Gamma+g)\}=\alpha\beta-b\beta\). In such a case, we know by \cite[Theorem 3.1]{almiyano} that \(c(\Delta_{[0,g]})=c(\Gamma)-a\alpha\). Therefore, the previous expression reads:
 		$$
 	\delta(\Delta_{[0,g]})=c(\Gamma)-a\alpha-\delta(\Gamma)+ab=\delta(\Gamma)-a\alpha+ab.
 	$$
 	
 	 Therefore
 	 \[
	 -W(g)<\delta(\Delta_{[0,g]})\Longleftrightarrow a\alpha-2ab<\delta(\Gamma)-a\alpha+ab\Longleftrightarrow 2a\alpha-3ab<\delta(\Gamma).
	 \]
 	 Since \(g=\alpha\beta-a\alpha-b\beta\) with \(1\leq a\leq\beta/2\) and \(1\leq b\leq \alpha/2\), it is easy to check that:
 	 \[
	 2a\alpha-3ab\leq \frac{\alpha\beta}{4}\leq\delta(\Gamma)=\frac{\alpha\beta-\alpha-\beta+1}{2}\Longleftrightarrow \alpha\beta-2(\alpha+\beta)+2\geq0.
	 \]
 	 Now, the last inequality is true if \(3\leq\alpha<\beta.\) Similarly, we obtain the same last inequality and hence the same final condition if we consider \(\mathrm{min}\{\Gamma\cap(\Gamma+g)\}=\alpha\beta-a\alpha\) instead of \(\mathrm{min}\{\Gamma\cap(\Gamma+g)\}=\alpha\beta-b\beta\).
 \medskip
 	
 	Hence, let us consider \(\Gamma=\langle\alpha,\beta\rangle\) with \(3\leq\alpha<\beta.\) Then, we have that any \(g\in\mathbb{N}\setminus\Gamma\) fulfills \(-W(g)<\delta(\Delta)\). Thus,
 	\[
	W_{\Delta_{[0,g]}}(3)=W(g)+\delta(\Delta_{[0,g]})\geq W(g)-W(g)=0.
	\]
 	
 	To finish the proof it remains to show that the statement of the theorem is also true for the case \(\Gamma=\langle 2,\beta\rangle\) with \(\gcd(\alpha,\beta)=1\). But for this case \(W(g)=0\) for all \(g\in\mathbb{N}\setminus\Gamma\) by \cite[Theorem 3.4]{almiyano-tetris}.
\end{proof}
\begin{rem}
	Let us clarify the assertion of the proof where we claim that if \(3\leq \alpha<\beta\) then \(\alpha\beta-2(\alpha+\beta)+2\geq 0\). Let us consider the real variable function \(F(x,y)=xy-2(x+y)+2\). 
	The function \(F(x,y)\) has a saddle point at \((2,2)\) and it is positive within the region 
	$$
	(x-2)(y-2)\geq2.
	$$
	
	 This can be easily seen if we look at the surface defined by \(F(x,y)-z=0\) is a hyperbolic paraboloid which is represented in Figure \ref{fig:silla}. 
	
	\begin{figure}[H]
		\includegraphics[scale=0.45]{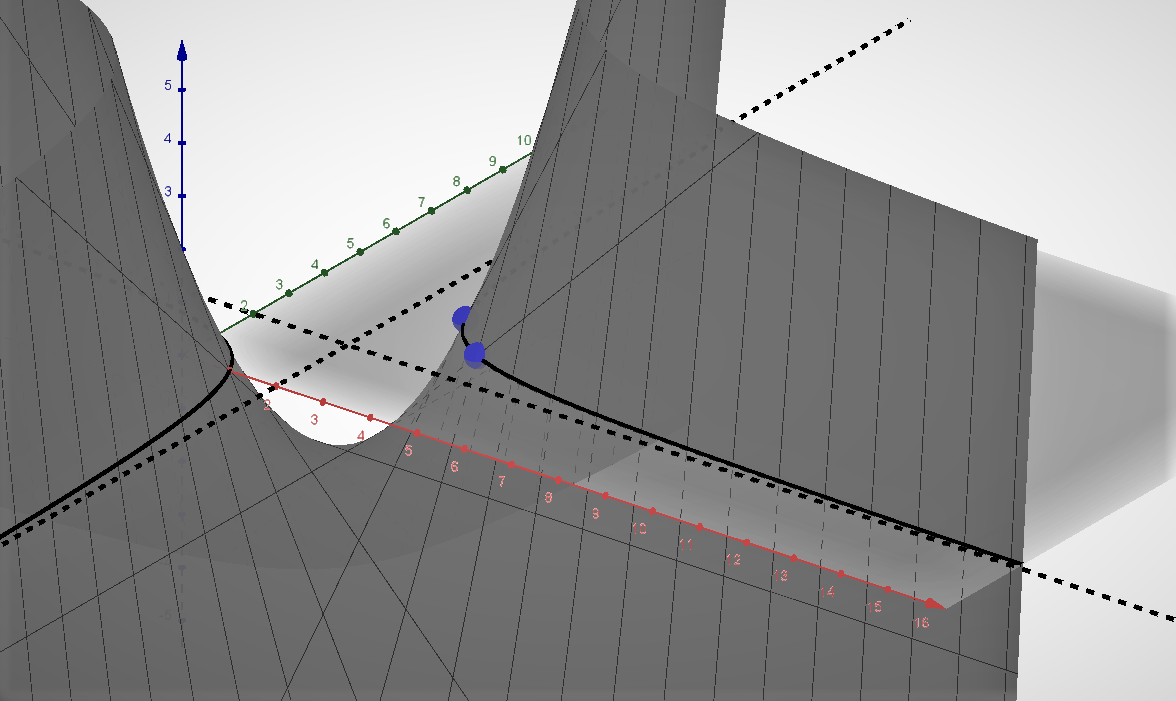}
		\caption{Representation of the surface \(xy-2(x+y)+2\).}
		\label{fig:silla}
	\end{figure}
	Since we are working with a numerical semigroup, we only need to take care about the values of \(F(x,y)\) with \(x\geq 0\), \(y\geq 0\) and \(x,y\in\mathbb{N}\). Observe that under those conditions the corresponding branch of the hyperbola \((x-2)(y-2)=2\) has only two points with \((x,y)\in\mathbb{N}^2\), namely $(3,4)$ and $(4,3)$, see Figure \ref{fig:sillaproy}. Moreover, since this hyperbola is symmetric with respect to bisection of the first positive quadrant, it is enough to look for points with \(x\leq y\). Therefore, any point \((x,y)\in\mathbb{N}^2\) such that \(3\leq x< y\) represents a semigroup \(\Gamma=\langle x,y \rangle\) satisfying the required inequality.
	
	\begin{figure}[H]
		\includegraphics[scale=0.45]{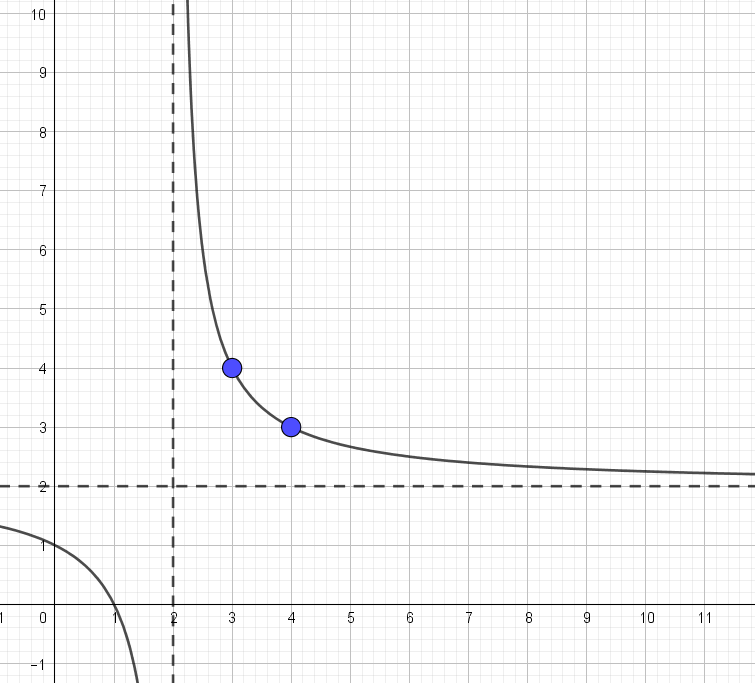}
		\caption{Section of the surface $xy-2(x+y)+2$ with the plane $z=0$.}
		\label{fig:sillaproy}
	\end{figure}
\end{rem}

Theorem \ref{thm:final} shows that the items in Question \ref{extensionwilfsemimodules} are far from being trivial. This is because ---even if we have a nice description for the Wilf function--- we need to exploit the constitutive properties of the gaps of the numerical semigroups, such as an adequate expression for them in terms of the minimal generators of the semigroup. It seems therefore to be a challenge to even guess the candidates for the invariants introduced in Question \ref{extensionwilfsemimodules}.
\medskip

In conclusion, we think that the issues presented in Question \ref{extensionwilfsemimodules} may offer a fruitful research topic to be investigated in the future.

\end{document}